\documentclass[10pt,a4paper]{article}

\usepackage[utf8]{inputenc}
\usepackage[english]{babel}
\usepackage{amsmath, amssymb, amsthm, amsfonts}
\usepackage{mathrsfs}
\usepackage{xcolor}
\usepackage{graphicx}
\usepackage{float}
\usepackage{hyperref}
\usepackage{geometry}
\usepackage{tikz-cd} 
\newtheorem{theorem}{Theorem}[section]
\newtheorem{corollary}[theorem]{Corollary}
\newtheorem{definition}[theorem]{Definition}
\newtheorem{lemma}[theorem]{Lemma}
\newtheorem{remark}[theorem]{Remark}
\newtheorem{proposition}[theorem]{Proposition}

\newcommand{\Sspace}{\mathcal{S}_{\psi,\omega}} 
\newcommand{\Sdist}{\mathcal{S}'_{\psi,\omega}} 
\newcommand{\WFT}{\mathcal{F}_{\psi,\omega}}     
\newcommand{\Trans}{\mathcal{T}}                 
\newcommand{\Sobolev}{H^s_{\psi,\omega}}         
\begin{document}
\title{Weighted Sobolev Spaces and Distributional Spectral \\ Theory for Generalized Aging Operators via Transmutation Methods.}

\author{
    \textbf{Gustavo Dorrego} \\[1ex]
    \small{Department of Mathematics, FACENA - UNNE} \\
    \small{Corrientes, Argentina} \\
    \small{\texttt{gadorrego@exa.unne.edu.ar}} 
}

\date{\today}

\maketitle

\begin{abstract}
    The spectral analysis of operators in heterogeneous and aging media typically requires a functional framework that extends beyond the standard Hilbertian setting. In this paper, we establish a rigorous distributional theory for a class of non-local operators, termed Weighted Weyl-Sonine operators, by employing a structure-preserving transmutation method. We construct the Weighted Schwartz Space $\mathcal{S}_{\psi,\omega}$ and its topological dual, the space of Weighted Tempered Distributions $\mathcal{S}'_{\psi,\omega}$, ensuring that the underlying Fréchet topology is consistent with the infinitesimal generator of the aging dynamics. This topological foundation allows us to: (i) extend the Weighted Fourier Transform ($\WFT$) to generalized functions as a unitary isomorphism; (ii) provide an explicit spectral characterization of the weighted Dirac delta $\delta_{\psi,\omega}$ and its scaling laws under geometric dilations; and (iii) introduce a scale of Weighted Sobolev Spaces $H^{s}_{\psi,\omega}$ defined via spectral multipliers. A central result is the derivation of a sharp embedding theorem, $|u(t)| \le C \omega(t)^{-1} ||u||_{H^s_{\psi,\omega}}$, which rigorously connects abstract spectral energy to the pointwise decay induced by the weight $\omega$. This framework provides a unified geometric characterization of several fractional regimes, including the Hadamard and Riemann-Liouville cases, within a single operator-theoretic architecture.

    \vspace{0.5cm}

\end{abstract}

\section{Introduction}

The spectral analysis of aging and heterogeneous systems has recently gained significant traction through the framework of Weighted Weyl-Sonine operators. In our previous work \cite{DorregoJMAA}, we established a Generalized Spectral Mapping Theorem, proving that the Weighted Fourier Transform $\WFT$ diagonalizes these non-local operators on the Hilbert space $L^2_{\psi,\omega}(\mathbb{R})$. That result provided a rigorous description of diffusive dynamics for signals with finite weighted energy, unifying models with power-law memory and those with multi-scale aging properties \cite{Metzler2000}.

However, the restriction to the Hilbertian setting $L^2_{\psi,\omega}$ imposes intrinsic limitations when dealing with idealized physical models and boundary value problems. Specifically:
\begin{itemize}
    \item \textbf{Singular Inputs:} Physical sources are often modeled as impulses (Dirac deltas) or abrupt steps, which possess infinite energy and thus lie outside the domain of the standard $L^2$-theory.
    \item \textbf{Growth at Infinity:} Signals that do not decay (e.g., constant background fields or polynomial growth) are excluded from the spectral formulation.
    \item \textbf{Regularity Constraints:} The rigorous treatment of boundary values and the quantification of "smoothing effects" in aging media require a scale of Sobolev spaces rather than a single energy space \cite{KufnerBook}.
\end{itemize}

\subsection{Main Contributions}
In this paper, we extend the spectral theory from the Hilbertian setting to the \textbf{distributional setting}. Following the philosophy of Strichartz \cite{StrichartzBook} and employing the transmutation methods of Shishkina \& Sitnik \cite{ShishkinaBook}, we construct a chain of topological spaces:
\begin{equation}
    \Sspace \subset L^2_{\psi,\omega} \subset \Sdist.
\end{equation}

This construction is not merely a formal generalization; it preserves the spectral structure of the aging operators. Our main results are:
\begin{enumerate}
    \item \textbf{Structure Transport:} We define the Weighted Schwartz Space $\Sspace$ and its dual $\Sdist$ via an isometric transmutation map, ensuring that the Weighted Fourier Transform extends to distributions as a unitary isomorphism consistent with Parseval's identity.
    \item \textbf{The Weighted Delta:} We derive the explicit form of the weighted Dirac delta, $\delta_{\psi,\omega} \sim (\omega^2 \psi')^{-1} \delta$, which reveals precisely how the medium's aging density and geometric dilation scale local impulses.
    \item \textbf{Sobolev Regularity:} We introduce the Weighted Sobolev Spaces $\Sobolev$ and prove a sharp embedding theorem, $|u(t)| \le C \omega(t)^{-1} \|u\|_{\Sobolev}$, which links abstract spectral energy to pointwise time-decay in the deformed medium.
\end{enumerate}

\subsection{Organization of the Paper}
The paper is organized as follows. Section 2 constructs the test function space $\Sspace$ based on the infinitesimal generator of the Weyl-Sonine operators. Section 3 develops the theory of Weighted Tempered Distributions and characterizes the weighted Dirac delta. Section 4 extends the Weighted Fourier Transform to distributions, justifying the spectral diagonalization of singular kernels. Finally, in Section 5, we define the Weighted Sobolev Spaces and establish the regularity and embedding theorems crucial for the analysis of aging diffusion equations.

\section{The Weighted Schwartz Space \texorpdfstring{$\Sspace$}{S\_psi,omega}}

Before formalizing the functional framework, we establish the class of admissible functions for the geometric scale $\psi$ and the weight $\omega$. Throughout this work, we assume the following:

\begin{enumerate}
    \item[\textbf{(H1)}] The scale function $\psi: \mathbb{R} \to \mathbb{R}$ is a $C^\infty$-diffeomorphism. Specifically, $\psi$ is strictly increasing, and $\psi'(t) > 0$ for all $t \in \mathbb{R}$.
    \item[\textbf{(H2)}] The weight function $\omega: \mathbb{R} \to (0, \infty)$ is a smooth function ($C^\infty$) such that $\omega(t) > 0$ for all $t$. Furthermore, we assume that $\omega$ and its derivatives possess at most polynomial growth to ensure that the multiplication operator is well-defined on $\mathcal{S}(\mathbb{R})$.
\end{enumerate}

These conditions guarantee that the transmutation operator $\Trans$ and its inverse $\Trans^{-1}$ are well-defined, smooth, and preserve the asymptotic behavior required for the Schwartz-type topology.

Following the spectral framework established in our previous work \cite{DorregoPaper1}, we construct the test function space via the canonical isometry that maps the deformed geometry to the Euclidean standard.

\begin{definition}[The Transmutation Operator]
    Let $\Trans: L^2_{\psi,\omega}(\mathbb{R}) \to L^2(\mathbb{R})$ be the unitary operator defined by:
    \begin{equation}
        (\Trans f)(y) := \omega(\psi^{-1}(y)) \cdot f(\psi^{-1}(y)).
    \end{equation}
    This operator acts as a "rectifier," transforming weighted signals into standard energy signals by compensating for the density $\omega$ and the geometric stretch $\psi$.
\end{definition}

\begin{definition}[Weighted Schwartz Space]
    In strict accordance with the spectral framework developed in \cite{DorregoJMAA}, we define the Weighted Schwartz Space $\Sspace$ as the set of smooth functions whose structure is preserved under the transmutation map.
    
    Explicitly, the topology is generated by the semi-norms induced by the \textbf{Fundamental Weighted Derivative} $\mathfrak{D}_{\psi,\omega}$, which serves as the infinitesimal generator of the Weyl-Sonine operators:
    
    \begin{equation}
        \rho_{k,m}(\phi) := \sup_{t \in \mathbb{R}} \left| \psi(t)^k \cdot \mathfrak{D}_{\psi,\omega}^m \phi(t) \right| < \infty, \quad \forall k, m \in \mathbb{N}_0,
    \end{equation}
    where the differential operator is defined as:
    \begin{equation}
        \mathfrak{D}_{\psi,\omega} \phi(t) := \frac{1}{\omega(t)\psi'(t)} \frac{d}{dt} \Big( \omega(t) \phi(t) \Big).
    \end{equation}
\end{definition}

\begin{remark}
    Note that we construct the semi-norms using the \textit{integer-order} differential operator $\mathfrak{D}_{\psi,\omega}$ rather than the non-local Weyl-Sonine integral operators. This ensures that $\Sspace$ is a Fréchet space of locally smooth functions, upon which the fractional Weyl-Sonine operators (defined in \cite{DorregoJMAA} as convolution multipliers) act continuously.
\end{remark}
\begin{proposition}[Density and Approximation]
    The space $\Sspace$ is dense in the Hilbert space $L^2_{\psi,\omega}(\mathbb{R})$. That is, for every square-integrable signal $f \in L^2_{\psi,\omega}$, there exists a sequence of test functions $\{\phi_n\}_{n=1}^\infty \subset \Sspace$ such that:
    \begin{equation}
        \lim_{n \to \infty} \| f - \phi_n \|_{L^2_{\psi,\omega}} = 0.
    \end{equation}
\end{proposition}

\begin{proof}
    The result follows immediately from the unitary property of $\Trans$. Since the classical Schwartz space $\mathcal{S}(\mathbb{R})$ is dense in $L^2(\mathbb{R})$, for any $g = \Trans f \in L^2(\mathbb{R})$, there exists a sequence $\{g_n\} \subset \mathcal{S}(\mathbb{R})$ converging to $g$. 
    Defining $\phi_n := \Trans^{-1} g_n$, we have $\phi_n \in \Sspace$ by definition, and the convergence is preserved isometrically:
    $$ \| f - \phi_n \|_{\psi,\omega} = \| \Trans^{-1}(g - g_n) \|_{\psi,\omega} = \| g - g_n \|_{L^2(\mathbb{R})} \to 0. $$
\end{proof}

\begin{remark}[Constructive Sequence: Weighted Hermite Functions]
    A concrete example of such an approximating sequence is given by the \textbf{Weighted Hermite Functions}. Let $h_n(y)$ be the classical Hermite functions (eigenfunctions of the standard Fourier transform). We define the weighted basis functions as:
    \begin{equation}
        \mathcal{H}_n^{\psi,\omega}(t) := (\Trans^{-1} h_n)(t) = \frac{1}{\omega(t)} h_n(\psi(t)).
    \end{equation}
    Since $\{h_n\}$ forms an orthonormal basis for $L^2(\mathbb{R})$ and $h_n \in \mathcal{S}(\mathbb{R})$, the collection $\{\mathcal{H}_n^{\psi,\omega}\}_{n \in \mathbb{N}_0}$ forms an orthonormal basis for $L^2_{\psi,\omega}$ contained entirely within $\Sspace$. Thus, the partial sum expansion of any signal $f$ converges in the weighted norm.
\end{remark}

\begin{theorem}[Isomorphism]
    The operator $\Trans$ is a topological isomorphism between $\Sspace$ and $\mathcal{S}(\mathbb{R})$. Consequently, the Weighted Fourier Transform satisfies $\WFT = \mathcal{F} \circ \Trans$, allowing us to import the entire machinery of classical harmonic analysis.
\end{theorem}

\begin{center}
\begin{tikzcd}[row sep=large, column sep=huge, ampersand replacement=\&]
    \Sspace \arrow[r, "\WFT"] \arrow[d, "\Trans"', "\cong"] 
    \& \Sspace \arrow[d, "\Trans", "\cong"'] \\ 
    \mathcal{S}(\mathbb{R}) \arrow[r, "\mathcal{F}"] 
    \& \mathcal{S}(\mathbb{R})
\end{tikzcd}
\end{center}
\begin{remark}[Topological Consistency]
    It is crucial to emphasize that since $\Trans$ is a unitary isomorphism, the structural and topological properties of $\Sspace$—including completeness, separability, and its Fréchet nature—are inherited directly from the classical Schwartz space. This "structure transport" approach bypasses the necessity for an independent construction of the underlying functional analysis, ensuring that the distributional framework developed in the subsequent sections is inherently well-posed and consistent with the Euclidean theory.
\end{remark}
\section{Weighted Tempered Distributions \texorpdfstring{$\Sdist$}{S'\_psi,omega}}

Having established the test function space, we now define generalized functions in the weighted setting via topological duality.

\begin{definition}[Distributions via Transmutation]
    The space of Weighted Tempered Distributions $\Sdist$ is defined as the topological dual of $\Sspace$. The duality pairing between a distribution $T \in \Sdist$ and a test function $\phi \in \Sspace$, denoted by $(T, \phi)_{\psi,\omega}$, is defined to ensure compatibility with the transmutation structure:
    \begin{equation}
        (T, \phi)_{\psi,\omega} := (\Trans^{-1} T, \Trans^{-1} \phi)_{L^2},
    \end{equation}
    where $(\cdot, \cdot)_{L^2}$ denotes the standard distributional pairing on $\mathbb{R}$.
\end{definition}

A critical consequence of this definition is the characterization of singular sources. In physical applications, impulsive inputs are modeled by the Dirac delta. The following Lemma clarifies how the medium's heterogeneity distorts these impulses.

\begin{lemma}[Representation of the Weighted Delta]
\label{lem:weighted_delta}
    Let $\delta_{\psi,\omega}(\cdot-\tau)$ be the weighted Dirac delta concentrated at $\tau \in \mathbb{R}$, defined by the canonical sampling property:
    \begin{equation}
        (\delta_{\psi,\omega}(\cdot-\tau), \phi)_{\psi,\omega} = \phi(\tau), \quad \forall \phi \in \Sspace.
    \end{equation}
    This distribution admits the explicit representation in terms of the standard Dirac delta $\delta$:
    \begin{equation}
        \delta_{\psi,\omega}(t-\tau) = \frac{1}{\omega(\tau)^2 \psi'(\tau)} \delta(t-\tau).
    \end{equation}
\end{lemma}

\begin{proof}
    We proceed formally by identifying the distribution with its action under the weighted measure $d\mu_{\psi,\omega}(t) = \omega(t)^2 \psi'(t) \, dt$. We seek a generalized function $D(t)$ such that:
    \begin{equation*}
        \int_{-\infty}^{\infty} D(t) \phi(t) \, \omega(t)^2 \psi'(t) \, dt = \phi(\tau).
    \end{equation*}
    Substituting the ansatz $D(t) = \frac{1}{\omega(\tau)^2 \psi'(\tau)} \delta(t-\tau)$ into the integral:
    \begin{align*}
        I &= \int_{-\infty}^{\infty} \left[ \frac{\delta(t-\tau)}{\omega(\tau)^2 \psi'(\tau)} \right] \phi(t) \, \omega(t)^2 \psi'(t) \, dt \\
          &= \frac{1}{\omega(\tau)^2 \psi'(\tau)} \int_{-\infty}^{\infty} \delta(t-\tau) \left[ \phi(t) \omega(t)^2 \psi'(t) \right] \, dt.
    \end{align*}
    By the sifting property of the standard Dirac delta, the integral evaluates the term in brackets at $t=\tau$:
    \begin{equation*}
        I = \frac{1}{\omega(\tau)^2 \psi'(\tau)} \cdot \left[ \phi(\tau) \omega(\tau)^2 \psi'(\tau) \right] = \phi(\tau).
    \end{equation*}
    This confirms that the scaling factor $(\omega^2 \psi')^{-1}$ is the necessary Radon-Nikodym derivative required to normalize the impulse against the medium's density.
\end{proof}

\begin{remark}[Geometric Dilution]
    The scaling factor $\mathcal{J}(\tau) = [\omega(\tau)^2 \psi'(\tau)]^{-1}$ has a profound physical interpretation. In regions where the medium is "dense" (large $\omega$) or where the effective time expands rapidly (large $\psi'$), the impact of a standard local impulse is diluted. Consequently, the source term must be renormalized (amplified) by the geometry of the space to achieve a unit effect on the spectral components.
\end{remark}

\section{The Distributional Weighted Fourier Transform}

In this section, we justify the application of the Weighted Fourier Transform $\WFT$ to singular kernels. This extension relies on the topological duality established in Section 3 and strictly follows the algebraic construction provided in the Appendix of our previous work \cite{DorregoJMAA}.

\subsection{Definition via Parseval's Identity}
Since $\WFT$ is a unitary isomorphism on the test space $\Sspace$, the standard Parseval identity $\int \widehat{u} v = \int u \widehat{v}$ holds. We use this symmetry to define the transform on distributions via duality.

\begin{definition}[Distributional Weighted Fourier Transform]
    Let $T \in \Sdist$ be a weighted tempered distribution. We define its Weighted Fourier Transform, denoted by $\widehat{T} = \WFT[T]$, as the unique functional in $\Sdist$ satisfying:
    \begin{equation}
        (\widehat{T}, \varphi)_{\psi,\omega} := (T, \WFT[\varphi])_{\psi,\omega}, \quad \forall \varphi \in \Sspace.
    \end{equation}
    Here, $(\cdot, \cdot)_{\psi,\omega}$ denotes the duality pairing defined in Definition 3.1.
\end{definition}

\begin{remark}[Consistency with Plancherel]
    This definition ensures that the spectral mapping property holds rigorously for singular convolution kernels. If $k \in \Sdist$ represents a generalized memory kernel, the convolution theorem extends as:
    \begin{equation}
        \WFT \left[ \mathfrak{D}_{\psi, \omega}^{(k)} u \right] = \widehat{k} \cdot \widehat{u},
    \end{equation}
    where the product on the right-hand side is understood as the multiplication of a distribution by a smooth spectral symbol (or via the exchange formula).
\end{remark}

\begin{theorem}[Spectral Mapping for Aging Operators]
    Specifically, for the Weyl-Sonine aging operators discussed in \cite{DorregoJMAA}, the kernel $k$ is associated with a Bernstein symbol $\Phi(i\xi)$. The distributional transform diagonalizes the operator algebraically:
    \begin{equation}
        \widehat{\mathfrak{D}_{\psi,\omega}^{(\Phi)} T}(\xi) = \Phi(i\xi) \cdot \widehat{T}(\xi).
    \end{equation}
\end{theorem}
\section{Weighted Sobolev Spaces \texorpdfstring{$\Sobolev$}{H\textasciicircum s\_psi,omega}}

With the distributional machinery in place, we quantify the regularity of solutions using energy norms. This allows us to move from formal operational calculus to rigorous analysis.

\begin{definition}[Spectral Definition]
    For any real number $s \in \mathbb{R}$, the Weighted Sobolev Space $\Sobolev$ is defined as the subspace of distributions $u \in \Sdist$ satisfying:
    \begin{equation}
        \|u\|_{\Sobolev}^2 := \int_{-\infty}^{\infty} (1+|\xi|^2)^s \left| \WFT[u](\xi) \right|^2 d\xi < \infty.
    \end{equation}
\end{definition}

\begin{proposition}[Completeness and Isometry]
    The space $\Sobolev$ is a Hilbert space. Furthermore, the transmutation operator restricts to a unitary isomorphism $\Trans: \Sobolev \to H^s(\mathbb{R})$, mapping weighted regularity directly to classical Sobolev regularity.
\end{proposition}

One of the most physically relevant results of this work is the following embedding theorem, which relates abstract spectral energy to the pointwise behavior of the signal in the aging medium.

\begin{theorem}[Weighted Sobolev Embedding]
\label{thm:embedding}
    Let $s > 1/2$. Then, there is a continuous embedding $\Sobolev \hookrightarrow C^0_{\psi,\omega}(\mathbb{R})$. Explicitly, if $u \in \Sobolev$, then $u$ corresponds to a continuous function satisfying the pointwise decay bound:
    \begin{equation}
        |u(t)| \leq \frac{C_s}{\omega(t)} \|u\|_{\Sobolev}, \quad \forall t \in \mathbb{R},
    \end{equation}
    where $C_s$ is the embedding constant of the classical space $H^s(\mathbb{R})$.
\end{theorem}

\begin{proof}
    Let $u \in \Sobolev$. Define $v = \Trans u$. By the proposition above, $v \in H^s(\mathbb{R})$. Since $s > 1/2$, the classical Sobolev Embedding Theorem implies $v \in C_0(\mathbb{R}) \cap L^\infty(\mathbb{R})$ with the bound:
    \begin{equation*}
        \|v\|_{L^\infty} \leq C_s \|v\|_{H^s}.
    \end{equation*}
    Recalling the explicit inversion formula from the transmutation definition, we have $u(t) = \frac{1}{\omega(t)} v(\psi(t))$. Taking the absolute value:
    \begin{equation*}
        |u(t)| = \frac{1}{|\omega(t)|} |v(\psi(t))| \leq \frac{1}{|\omega(t)|} \|v\|_{L^\infty} \leq \frac{C_s}{|\omega(t)|} \|v\|_{H^s}.
    \end{equation*}
    Since $\|v\|_{H^s} = \|u\|_{\Sobolev}$ by isometry, the result follows.
\end{proof}

This embedding result is visually illustrated in Figure \ref{fig:envelope}, showing how the density acts as a physical envelope for the signal.

\begin{figure}[H]
    \centering
    \includegraphics[width=0.7\textwidth]{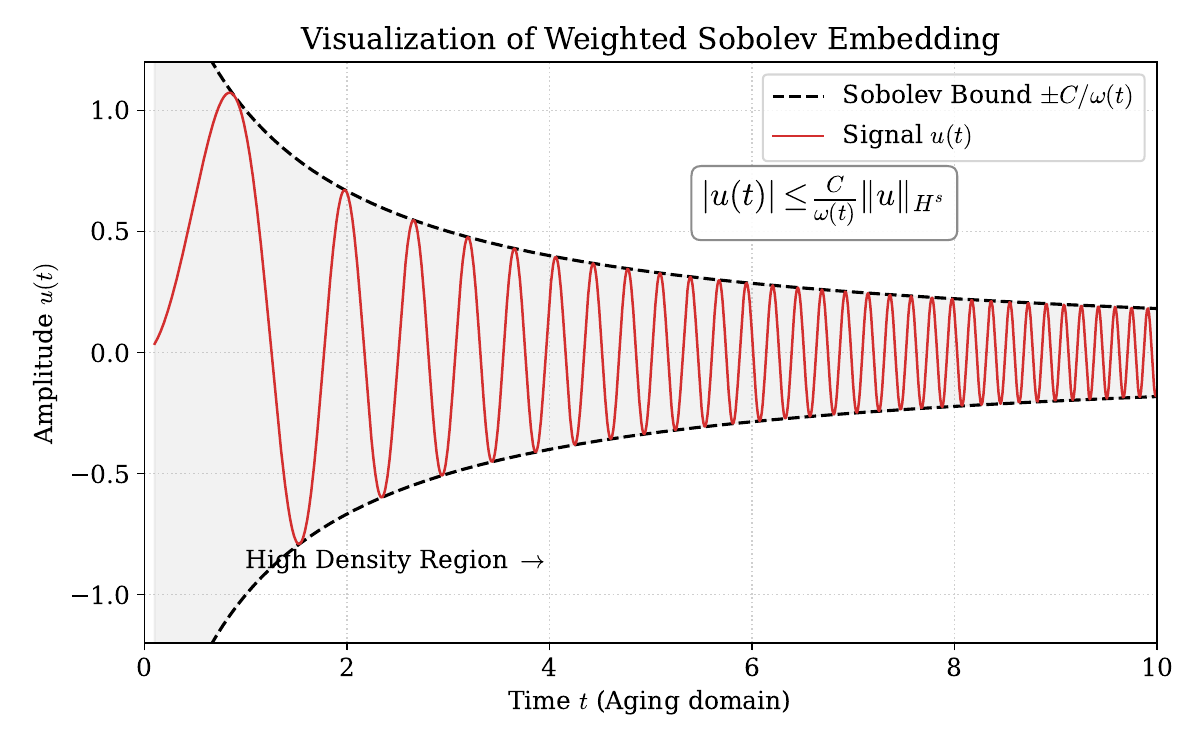}
    \caption{\textbf{Physical interpretation of the Weighted Sobolev Embedding.} The function $u(t)$ (red) represents a finite-energy signal in the aging medium. As the medium density $\omega(t)$ increases, the signal is physically constrained to decay within the envelope $\pm C_s/\omega(t)$ (dashed lines), regardless of its frequency. This illustrates the "amplitude suppression" mechanism inherent to the weighted topology.}
    \label{fig:envelope}
\end{figure}

\begin{theorem}[Elliptic Regularity and Isomorphism]
\label{thm:regularity}
    Let $\alpha \in (0, 2)$. Consider the fractional aging equation with spectral shift:
    \begin{equation}
        \mathfrak{D}_{\psi,\omega}^{(\alpha)} u + u = f.
    \end{equation}
    For any source term $f \in H^s_{\psi,\omega}$ (including singular distributions like $\delta_{\psi,\omega}$ if $s$ is low enough), the solution exists uniquely and gains $\alpha$ degrees of regularity:
    \begin{equation}
        u \in H^{s+\alpha}_{\psi,\omega}.
    \end{equation}
    Moreover, the operator $(\mathfrak{D}_{\psi,\omega}^{(\alpha)} + I)$ is a topological isomorphism between $H^{s+\alpha}_{\psi,\omega}$ and $H^s_{\psi,\omega}$.
\end{theorem}
\begin{proof}
    Applying the distributional Weighted Fourier Transform, the differential equation maps to the algebraic equation:
    \begin{equation*}
        \left( (i\xi)^\alpha + 1 \right) \widehat{u}(\xi) = \widehat{f}(\xi).
    \end{equation*}
    Since $\alpha \in (0, 2)$, the angular range of $(i\xi)^\alpha$ is strictly contained in $(-\pi, \pi)$, implying that the symbol $\sigma(\xi) = (i\xi)^\alpha + 1$ never vanishes on $\mathbb{R}$. Moreover, since $\sigma(\xi) \to 1$ as $\xi \to 0$ and $|\sigma(\xi)| \to \infty$ as $|\xi| \to \infty$, there exists a constant $c > 0$ such that $|\sigma(\xi)| \ge c$. Thus, the inverse operation is well-defined. We estimate the norm:
    \begin{equation*}
        \|u\|_{H^{s+\alpha}_{\psi,\omega}}^2 = \int_{-\infty}^{\infty} (1+|\xi|^2)^{s+\alpha} \frac{|\widehat{f}(\xi)|^2}{|(i\xi)^\alpha + 1|^2} \, d\xi.
    \end{equation*}
    The multiplier function $M(\xi) = \frac{(1+|\xi|^2)^{\alpha/2}}{|(i\xi)^\alpha + 1|}$ is continuous and bounded on $\mathbb{R}$ (behaving as $1$ near $\xi=0$ and approaching $1$ as $|\xi|\to\infty$). Therefore, there exists a constant $C>0$ such that:
    \begin{equation*}
        \|u\|_{H^{s+\alpha}_{\psi,\omega}}^2 \leq C \int_{-\infty}^{\infty} (1+|\xi|^2)^s |\widehat{f}(\xi)|^2 \, d\xi = C \|f\|_{H^s_{\psi,\omega}}^2.
    \end{equation*}
    This proves the regularity gain and the continuous dependence on the data.
\end{proof}

\begin{corollary}[Impulse Response and Decay Envelope]
Consider the impulsive aging equation $(\mathfrak{D}_{\psi,\omega}^{(\alpha)} + I) \mathcal{G}_{t_0} = \delta_{\psi,\omega}(t - t_0)$ with $\alpha > 1$. The fundamental solution $\mathcal{G}_{t_0}(t)$ belongs to $H^{s+\alpha}_{\psi,\omega}$ for $s < -1/2$. By the Sobolev embedding (Theorem \ref{thm:embedding}), this solution is a continuous function satisfying:
\begin{equation}
    |\mathcal{G}_{t_0}(t)| \leq \frac{C}{\omega(t) \omega(t_0)}, \quad \forall t \in \mathbb{R}.
\end{equation}
This confirms that the physical weight $\omega$ acts as a bilateral constraint, scaling both the source's impact at $t_0$ and the signal's decay at $t$.
\end{corollary}

\begin{remark}
    The shift term "$+u$" is necessary to ensure invertibility at zero frequency (where the pure operator $\mathfrak{D}_{\psi,\omega}^{(\alpha)}$ is singular). This reflects the fact that $H^s_{\psi,\omega}$ spaces model Bessel-potential type regularity rather than Riesz-potential type homogeneity.
\end{remark}

\begin{remark}[Unified Geometric Characterization: The Hadamard Case]
Our framework recovers the classical Hadamard fractional calculus by setting $\psi(t) = \ln(t)$ for $t > 0$ and $\omega(t) = 1$. In this case, the embedding theorem predicts a constant bound (since $1/\omega = 1$), consistent with the logarithmic growth of the underlying geometry. This demonstrates how the structural pair $(\psi, \omega)$ serves as a dictionary to translate regularity results across different fractional regimes.
\end{remark}
\section{Conclusions}

In this work, we have successfully extended the unified spectral framework initiated in \cite{DorregoJMAA} to the realm of generalized functions. By employing the "Structure Transport" method via transmutation operators, we avoided ad-hoc definitions, constructing the Weighted Schwartz Space $\Sspace$ and its dual $\Sdist$ in a manner intrinsically compatible with the aging geometry.

Our analysis revealed a dual geometric mechanism imposed by the medium's heterogeneity: while the density $\omega(t)$ dilutes the impact of singular sources (scaling the Dirac delta as $\delta_{\psi,\omega} \sim (\omega^2 \psi')^{-1} \delta$), it simultaneously enforces a strict decay envelope on regular signals via the Weighted Sobolev Embedding ($|u| \lesssim \omega^{-1}$).

These results provide the necessary functional analytic foundation to address well-posedness questions for non-linear evolution equations. Future investigations will focus on the associated Cauchy problem via semigroup theory and the numerical implementation of these spectral operators for singular control problems.

\end{document}